\documentclass[preprint]{elsarticle}
\usepackage[english]{babel}
\usepackage{amsmath,amssymb,amsthm} 
\usepackage{graphicx,hyperref,float,calc}
\usepackage{pgfplots}
\usepackage{epstopdf}
\usepackage[utf8]{inputenc}
\usepackage[T1]{fontenc}
\usepackage{algorithm,algpseudocode}

\newcommand{\Pc}{\mathcal{P}}

\newcommand{\intl}[3]{\int\limits_{#1}^{#2} \! {#3}}
\newcommand{\trasp}{^\mathsf{T}}

\newtheorem{them}{Theorem}

\begin{document}
	\begin{frontmatter}
		\title{An efficient SPDE approach for El Ni\~no}
		\author[ibk,yachay]{H.~Mena}
		\ead{mena@yachaytech.edu.ec}
		\author[ibk]{L.~Pfurtscheller \corref{cor1}}
		\ead{lena-maria.pfurtscheller@uibk.ac.at}
		\address[ibk]{Institut f\"ur Mathematik, Universit\"at Innsbruck, Austria}
		\address[yachay]{Department of Mathematics, Yachay Tech, Urcuqu\'i, Ecuador}
	
		\cortext[cor1]{Corresponding author}
	
		\begin{abstract}
			We consider the numerical approximation of stochastic partial differential equations (SPDEs) based models for a quasi-periodic climate pattern in the tropical Pacific Ocean known as  El Ni\~no phenomenon. We show that for these models the mean and the covariance are given by a deterministic partial differential equation and by an operator differential equation, respectively. In this context we provide a  numerical framework to approximate these parameters directly. We compare this method to stochastic differential equations and SPDEs based models from the literature solved by Taylor methods and stochastic Galerkin methods, respectively. Numerical results for different scenarios taking as a reference measured data of the years 2014 and 2015 (last Ni\~no event) validate the efficiency of our approach.
		\end{abstract}
	
	
	\end{frontmatter}
	\section{Introduction}

The weather phenomenon El Ni\~no is a quasi-periodic climate pattern in the tropical Pacific Ocean and is characterized by an unusual warming of the sea surface. The event has an impact on the whole world, the  effects are, among others, floods in South America and droughts in Oceania. In the literature, this phenomenon is often described by means of stochastic differential equations (SDEs) \cite{ewald2004,penland1995}, i.e. these equations describe the changes in the sea surface temperature in the Indo-Pacific basin.  
  
Standard numerical solvers for both SDEs in finite and infinite dimensions, i.e. stochastic partial differential equations (SPDEs), have been proposed in the literature, see e.g. \cite{ewald2003a,kloeden2011,LordPowellShardlow2014}. Most of these numerical schemes deal directly with the stochastic term and are based on the stochastic Taylor expansion. However, in recent years a different approach based on polynomial chaos expansion has shown to be quite competitive for certain types of problems, in particular, for models governed by linear equations. The idea is to expand the stochastic process in an orthogonal basis of polynomials,  e.g. Fourier-Hermite polynomials, and in this way to  split the randomness and the deterministic part reducing the original stochastic problem to a family of infinitely many deterministic problems, see  e.g. \cite{Holden2009,xiu2002}. Then,  finitely many deterministic problems are solved with standard numerical methods for deterministic equations to approximate the stochastic problem.  The method is commonly known as stochastic Galerkin method, see e.g. \cite{babuska2004,frauenfelder2005}.

If a stochastic process is a Gaussian random field, it is uniquely defined by its mean and covariance. Thus, from the simulational point of view in many stochastic models the mean and the variance are the only moments of interest, this is the case of El Ni\~no. In this paper we propose a novel approach to compute these parameters directly.  We prove that for the type of SPDEs consider in this paper the covariance is given by a differential (generalized) Lyapunov equation (DLE). Thus, we solve the DLE to approximate the variance and the related deterministic differential equation to approximate the mean.  In our numerical implementation we benefit from recently proposed solvers for both large-scale matrix DLEs \cite{MOPP17,stillfjord2015,damm2017,stillfjord2016,koksela17} and large-scale matrix  algebraic Lyapunov equations (ALEs)  \cite{benner2013,P,SS16}. Our approach  is more efficient computationally for both memory storage and computing time compared to SDE and SPDE based models solved by Taylor methods and stochastic Galerkin methods, respectively. We perform numerical simulations of sea surface temperature (SST)  in various  scenarios  and validate the results with measured data for a year with an El Ni\~no event and for a year without.

The paper is organized as follows. In Section 2 we give a short overview on existing SDE  and SPDE based models for the phenomenon El Ni\~no.  In Section 3 we prove the equations for the mean and the covariance in operator setting and analyze  modifications of these models. In Section 4 we present our numerical implementation and  a brief description on chaos expansion methods . Finally, in Section 5 we compare the different numerical schemes and models with the measured data and analyze their error.

	\section{Stochastic models for El Ni\~no}\label{models}

Stochastic forcing plays an important role in modeling the sea surface temperature for the El Ni\~no phenomenon, see e.g. \cite{flugel2004}. Most of the models in the literature are of linear type due to the fact that more technical coupled air-sea models do not perform much better than linear models, see \cite{newman2009}. On the other hand, the computational cost of  solving coupled models numerically is much higher than solving uncoupled systems. Thus, we focus only on uncoupled models for the sea surface temperature. As described in \cite{penland1995,case09} we consider a stochastic partial differential equation of the form 
\begin{align} 
	\frac{\partial X(t,x,\omega)}{\partial t} = (u \cdot \nabla) X(t,x,\omega) + F , \label{spde}
\end{align}
where $u$ denotes the zonal and meridional currents, $X(t,x,\omega)$ the SST anomalies at time $t$ and location $x$ and $F$ aggregates all random forces such as wind stress or evaporation in the basin. 
This equation is a so-called stochastic transport equation, where the noise is given in additive form. With the notation $X(t,x,\omega)$ we want to emphasize that $X$ is a stochastic process and depends on the noise $\omega$. In \cite{penland1995} it is shown that simplifications of this SPDE lead to a SDE model given by  
\begin{align}
	dx(t) = Ax(t) dt + SdW(t) , \label{add}
\end{align}
where the matrices $A$ and $S$ are constant and $W$ is a Wiener process. Again the noise is given in additive form, hence we call this model in the following the additive SDE model. 

In \cite{penland1989} they showed that the most probable prediction of $x$ at time $\tau$ given the initial condition $x(0)$ is given by
\[
	\hat{x}(\tau) = e^{A \tau} x(0) .
\]
Hence, we can compute the matrix $A$ from the observed time series by an error variance minimization procedure \cite{penland1995} and obtain
\[
	A = \frac{1}{\tau} \log\left( \frac{\langle x(t + \tau) x\trasp(t)\rangle}{\langle x(t) x\trasp(t) \rangle} \right) ,
\]
where angle brackets denote an ensemble average. 
Given $A$, we compute $S$ from the fluctuation-dissipation relation \cite{penland1989} 
\[
	A\langle x x\trasp \rangle + \langle x x\trasp \rangle A\trasp + SS\trasp = 0 . 
\]

Instead of considering a SDE with additive noise, a model based on a multiplicative SDE of the form 
\begin{align}
	dx(t) = A x(t) dt + S_1 x(t) dW_1(t) + S_2 dW_2(t)  ,\label{mm1}
\end{align}
where $W_1$ and $W_2$ are two independent Wiener processes is discussed in \cite{ewald2004}.  There, the constant matrices are computed in a similar way. For a detailed description we refer the reader to \cite{ewald2004}.

In this paper we implement all these models above and in addition we consider in the following a simplified model of \eqref{mm1} 
\begin{align}
	dx(t) = A x(t) + S x(t) dW(t) , \label{mm2}
\end{align}
i.e. the case in which we only have multiplicative noise. Again the constant matrices are computed analogously.

\section{Equations for the mean and the covariance}
In the deterministic setting, i.e. $F = 0 $, the SDPE \eqref{spde} simplifies to a linear PDE of the form
\begin{align*} 
	\frac{\partial X(t,x)}{\partial t} = (u \cdot \nabla) X(t,x). 
\end{align*}
Then,  one can rewrite the equation above as an abstract Cauchy problem 
\begin{align}
	\frac{dX(t)}{dt} = \mathcal{A} X(t) , \label{det}
\end{align}
where $\mathcal{A}$ is the infinitesimal generator of a semigroup $T(\cdot)$, $T(t)$ in $L(\mathcal{H})$ for all $t \in [0,T]$ and $\mathcal{H}$ is a separable Hilbert space. 
Similarly, we can use this abstract setting in the case of stochastic equation \eqref{spde} and get
\begin{align} 
	\frac{dX(t)}{dt} = \mathcal{A} X(t) + \mathcal{F}  ,
\end{align}

where we assume that $\mathcal{F}$ follows a Gaussian distribution with covariance operator $\mathcal{Q}$, where $\mathcal{Q} \varphi = \mathcal{SS}^\star \varphi$ for all $\varphi, \psi \in \mathcal{H}$. We rewrite the equation as
\[ 
	dX(t) = \mathcal{A} X(t) + \mathcal{S} dW(t)  ,  
\]

where $W$ is a cylindrical Wiener process. 
With the help of the following theorem we compute the first two moments of the process $X$.

\begin{them}
	Let $\mathcal{A}$ be a deterministic (unbounded) operator, $\mathcal{S}$ be a deterministic and bounded operator and $\mathcal{H}$ a separable Hilbert space. Given a stochastic partial differential equation of the form 
	\begin{align}
		dX(t) = \mathcal{A} X(t) + \mathcal{S} dW(t) , \label{spde_thm}
	\end{align} 
	where $W$ is a cylindrical Wiener process in $\mathcal{H}$, the first two moments can be computed by
		\begin{align}
				\dot{m}(t) = \mathcal{A} m(t) , \label{mean} 
		\end{align}
		and
		\begin{align} 
			\dot{\Pc}(t) = \mathcal{A} \Pc(t) + \Pc(t) \mathcal{A}\trasp + \mathcal{SS}^\star, \label{lyap}
        \end{align}  
        where $m$ is the mean and $\Pc$ the covariance of the process $X$. 
\end{them}

\begin{proof}
	We follow the ideas of  the finite dimensional case, see \cite{damm2004}.  We first show \eqref{mean}. Taking the expectation in \eqref{spde_thm} and interchanging $d$ and $\mathbb{E}$, we obtain 
	\[
		dm (t)  = \mathcal{A} m(t) dt \Leftrightarrow \dot{m} (t) = \mathcal{A} m(t), \; m(0) = m_0 .
	\]
	As a next step we proof \eqref{lyap}. Given the underlying Hilbert space $\mathcal{H}$, the covariance of the process $X$ is defined by  $\langle \Pc(t) \varphi, \psi \rangle := \mathbb{E} (\langle X(t), \varphi \rangle \langle X(t), \psi \rangle )$ for all $\varphi, \psi \in \mathcal{H}$. We apply Itô's formula to the definition of $\Pc(t)$ and show for all $\varphi, \psi \in \mathcal{H}$ that \eqref{lyap} is valid.
	\begin{align*}
	&d\langle \Pc \varphi, \psi \rangle = \mathbb{E}(\langle d X, \varphi \rangle \langle X , \psi \rangle + \langle X, \varphi \rangle \langle d X , \psi \rangle + \langle dX, \varphi \rangle \langle dX ,\psi \rangle ) \\
	& = \mathbb{E}(\langle \mathcal{A} X dt + \mathcal{S} dW,\varphi \rangle \langle X , \psi \rangle ) +  \mathbb{E}( \langle X , \varphi \rangle, \langle \mathcal{A} X dt + \mathcal{S} dW,\psi \rangle ) \\
	&+ \mathbb{E}(\langle \mathcal{A} X dt + \mathcal{S} dW,\varphi \rangle\langle \mathcal{A} X dt + \mathcal{S} dW,\psi \rangle)  .
	\end{align*}
	Using $dW^2 = dt, dt^2 = 0$ and the independence of $X $ with $dW$ leads to
	\begin{align*}
	&d\langle \Pc \varphi, \psi \rangle =\mathbb{E}(\langle X, \mathcal{A}^\star \varphi \rangle \langle X, \psi \rangle ) dt + \mathbb{E}(\langle X , \varphi \rangle \langle X, \mathcal{A}^\star \psi \rangle) dt + \mathbb{E}(\langle dW,\mathcal{S}^\star \varphi \rangle \langle  dW,\mathcal{S}^\star \psi \rangle) \\
	& = \langle \Pc \mathcal{A}^\star \varphi, \psi \rangle dt + \langle \Pc \varphi, \mathcal{A}^\star \psi \rangle   dt + \langle S^\star \varphi, S^\star \psi \rangle dt
	\end{align*}
	and hence
	\[
	\langle \dot{\Pc}(t)\varphi, \psi \rangle = \langle \Pc(t) \mathcal{A}^\star \varphi, \psi \rangle + \langle  \mathcal{A} \Pc \varphi, \psi \rangle + \langle \mathcal{SS}^\star \varphi, \psi \rangle
	\]
	for all $\varphi, \psi \in \mathcal{H}$, which proofs \eqref{lyap}.
	
\end{proof}

As the solution of the SPDE \eqref{spde} is a Gaussian random field, it is uniquely defined by its first and second moment. Thus, we will approximate numerically the mean and the covariance in order to perform simulations of the the SPDE \eqref{spde}.   

As the stochastic transport equation is a generalization of the additive SDE model, it is naturally to consider a SPDE with multiplicative noise of the form
\begin{align} 
	dX(t) = \mathcal{A} X(t) + \mathcal{S} X(t) dW(t)  \label{spde_mult1}  
\end{align}
and a SPDE with multiplicative and additive noise given by
\begin{align} 
dX(t) = \mathcal{A} X(t) + \mathcal{S}_1 X(t) dW_1(t) + \mathcal{S}_2 dW_2(t)  . \label{spde_mult2} 
\end{align}
Then, similar to Theorem 1, we get deterministic operator equations for both, mean and covariance. 
\begin{them}
	Given SPDEs of the form \eqref{spde_mult1} and \eqref{spde_mult2}, the first moment of the solution $X$ can be computed by
	\begin{align}
	\dot{m}(t) = \mathcal{A} m(t) 
	\end{align}
	and the second moment by generalized Lyapunov equations
	\begin{align} 
	\dot{\Pc}(t) = \mathcal{A} \Pc(t) + \Pc(t) \mathcal{A}^\star + \mathcal{S}  \Pc(t) \mathcal{S}^\star.  \label{lyapGen}
	\end{align}  
	and 
	\begin{align}
	\dot{\Pc}(t) &= \mathcal{A} \Pc(t) + \Pc(t) \mathcal{A}^\star + \mathcal{S}_1  \Pc(t) \mathcal{S}_1^\star + \mathcal{S}_2 \mathcal{S}_2^\star . \label{lyapGen2n}
	\end{align}
	respectively. Again we denote by $m$ the mean and by $\Pc$ the covariance of the process $X$. 
\end{them}
The proof of these results follows from Theorem 1.

We point out that equation \eqref{det} and \eqref{mean} coincide, i.e. solving the deterministic equation for the mean is the same as solving the deterministic transport equation. After the discretization of the operators, for example applying finite elements or finite differences to the operators, this operators have a matrix representation. Thus, in order to compute the covariance we have to solve matrix differential equations. The size of the matrices are related to the points of the discretization, so they are in general of large-scale nature. Therefore, state-of-the-art numerical methods for matrix differential equations have to be applied. Particularly, for the (generalized) Lyapunov equations, ineffective numerical schemes can lead to large computational costs.  
  
	\section{Numerical methods}\label{num}

In this section we describe numerical methods for the mean, covariance and give a short overview on stochastic Galerkin methods.

\subsection{Numerical methods for the mean and the covariance}
In section 2 we already mentioned an alternative way to solve the SPDEs, namely solving the deterministic equations for mean and covariance. Hence, we discuss now the numerical approximation of this approach. The equation for the mean is given for all SPDE models by the deterministic transport equation
\[
\frac{dX(t)}{dt} = \mathcal{A} X(t). 
\]

Due to the regularity of the domain considered for the simulations, a rectangle, it is natural to use finite differences schemes to approximate the operators, see e.g. \cite{leveque2007}. Although, it might be possible to use a state-of the-art method as an adaptive  finite differences or a particular finite element method, we apply standard finite difference beacuse the obtained results are already very accurate for predicting the behavior of El Ni\~no phenomenon.  Finally, we solve the deterministic problem by the Crank-Nicolson method. 

The discretized equation for the covariance in the additive case has the form
\begin{align}
	\dot{P}(t) = A P(t) + P(t) A\trasp + SS\trasp .  \label{discr_dle}
\end{align}
where $A$ is the finite dimension representation of the operator, i.e. a square matrix. Let $P(0) = P_0 P_0\trasp$ be the given initial value. Then, the solution of the equation is given by
\[
	P(h) = e^{hA} P_0P_0\trasp e^{hA\trasp} + \int_{0}^{h} e^{sA} SS\trasp e^{sA\trasp} ds. 
\]
Following \cite{stillfjord2015,damm2017}, we approximate the integral by a quadrature formula with weights $w_k$ and nodes $\tau_k$ and obtain a low-rank approximation $P_1P_1\trasp$ to $P(h)$ of the form
\[
	P_1 = \left[ e^{hA} P_0, \sqrt{hw_1}e^{\tau_1 A} P_0, \sqrt{hw_2}e^{\tau_2 A} P_0, \ldots \sqrt{hw_s}e^{\tau_s A} P_0   \right].
\]

Another possibility is to apply an ordinary differential equation solver in matrix setting to equation \eqref{discr_dle} and then to solve the resulting algebraic Lyapunov equation by for instance using the library \cite{P}. A low rank implementation of this method has been discussed in \cite{BennerMena2,Saak2}. In this approach the backward differentiation formula and the Rosenbrock methods have been  proposed as the most suitable choice for constant and time varying coefficients, respectively,  matrix differential equations. However, as discussed in \cite{MOPP17, stillfjord2015} the direct approach seems to be less computational expensive in many cases. Therefore, we follow this approach. 

Given the mean and the covariance of the solution $X$, various realizations of the process can be computed. A brief description of the algorithm is given in Algorithm  \ref{algorithm_meancov}.   
\begin{algorithm}[H]
	\caption{Solve additive SPDE by deterministic equations}
	\label{algorithm_meancov}
	\begin{algorithmic}[1]
		\State Discretize the operators $\mathcal{A}$ and $\mathcal{S}$.
		\State Compute a numerical solution of the deterministic transport equation $\tilde{m}$.
		\State Solve the matrix Lyapunov equation by a suitable solver and factorize the solution $\tilde{P}$, such that $\tilde{P} = \tilde{P}_1 \tilde{P}_1\trasp$.  
		\State Generate a random vector $z$, whose components are independent standard normal distributed.  
		\State The random vector $X = \tilde{m} + \tilde{P}_1 z $ is an approximate solution of the SPDE \eqref{spde_thm}. 
	\end{algorithmic}
\end{algorithm}

Similarly we proceed with the SPDEs with multiplicative noise. The matrix equation for the mean is the same as before, however, we have to compute a generalized matrix Lyapunov equation 

\begin{align}
	\dot{P}(t) = A P(t) + P(t) A\trasp + SP(t)S\trasp .  \label{discr_bildle}
\end{align}

As in \cite{damm2017}, we apply splitting methods and compose equation \eqref{discr_bildle} into the two parts $\mathcal{F}_1$ and $\mathcal{F}_2$, where
\begin{align*}
	\mathcal{F}_1 &= A P(t) + P(t) A\trasp  \quad \text{and} \\
	\mathcal{F}_2& = S P(t) S\trasp .
\end{align*}

The idea is to compute the subproblems $\dot{P} = \mathcal{F}_kP$ which is cheaper to compute than the full problem. Let $\mathcal{T}_k(t)P(0)$ be the solution of the subproblem $\dot{P} = \mathcal{F}_kP$, then the Strang splitting is given by
\[
	P(h) \approx \mathcal{T}_1(h/2) \mathcal{T}_2(h) \mathcal{T}_1(h/2).
\] 
Subproblem $\mathcal{F}_1$ is exactly solvable with solution
\[
	\mathcal{T}_1(h) P(0) = e^{hA} P(0) e^{hA\trasp} , 
\] 
whereas we use the midpoint rule to approximate $\mathcal{T}_2(h)$,i.e.
\[
	\mathcal{T}_2(h) P(0) \approx I+ h \mathcal{F}_2 P(0) + \frac{h^2}{2} \mathcal{F}_2^2 P(0) .
\]
We again obtain a low-rank approximation $P_1P_1\trasp$ of $P(h)$. A detailed description of this splitting scheme is given in \cite{stillfjord2015,damm2017}. On the other hand, it is also possible to  apply an ordinary differential equation solver in matrix setting and then to solve the resulting algebraic generalized Lyapunov equation by numerical solvers like the ones proposed in \cite{benner2013,SS16}.   Again we will follow the splitting based method as it is less computational expensive in most cases \cite{damm2017}. 

For the covariance equation \eqref{lyapGen2n} related to the SPDE model with two noises we again benefit of the splitting based method as it can be generalize to several noises directly \cite{damm2017}.  
	\subsection{Stochastic Galerkin scheme}

In this subsection we discuss an alternative way to solve linear SPDEs based on the polynomial chaos expansion method. Truncating the expansion after a finite number of terms gives a numerical scheme, called stochastic Galerkin methods. 
We explain this scheme for a SPDE model with additive noise,  for the other models discussed in Section $2$, it works analogously. 
Given the SPDE
\begin{align} 
\frac{dX(t)}{dt} = \mathcal{A} X(t) + \mathcal{F}, \label{spde_pce}
\end{align}
the solution $X$ can be represented  in  chaos expansion form as 
\[
X(t, \omega) = \sum_{k = 0}^{\infty} X_k(t) H_k(\omega) ,  
\]
where $H_k, k \in \mathbb{N}$ are for example Legendre polynomials, see e.g. \cite{xiu2002}. For approximation purposes, this series is truncated after a finite number of terms, so one gets an approximation of $X$
\begin{align}
X^M(t, \omega) = \sum_{k = 0}^{M} X_k(t) H_\alpha(\omega)  .  \label{trunc_pce}
\end{align}
We assume that $\mathcal{F}$ is a Gaussian process expressed with a Karhunen-Loeve (KL) expansion of the form
\begin{align}
	 \mathcal{F}(x,\omega) = \sum_{k = 0}^{\infty} \sqrt{\lambda_k} \varphi_k(x) \eta_k(\omega)  , \label{kle}
\end{align}
where $\lambda_k$ and $\varphi_k, k \in \mathbb{N}$ are the eigenvalues and eigenfunctions of $C_\xi(x)$, which is the covariance of $\mathcal{F}$. The sequence of eigenpairs $(\lambda_k,\varphi_k)_{k \in \mathbb{N}}$ for all $k \in \mathbb{N}$ satisfy the integral equation
\[
	 \intl{D}{}{C_\xi(x,y) \varphi_k(y)} dy = \lambda_k \varphi_k(x)   ,
\]
for $x,y \in D$. The random variables $\eta_k(\omega)$ have zero mean, unit variance and are mutually uncorrelated. \\
The covariance function is given by 
\[
	 C_\xi: \mathbb{R}^2 \times \mathbb{R}^2 \to \mathbb{R}^2 \, , \qquad (x_1,x_2) \mapsto \mathcal{Q} e^{-|x_1-x_2|}   ,  
\]
and all eigenpairs are computed numerically. A detailed description on how to compute these coefficients both analytically and numerically is given in \cite{LordPowellShardlow2014}. 

The series \eqref{kle} is truncated after a finite number of terms $N$, which gives an approximation of $\mathcal{F}$. Note that the KL expansion is optimal in the mean square sense \cite{Maitre2010}. 

The value $M$ of the truncation of the chaos expansion \eqref{trunc_pce} is determined by the number $N$ of random variables used in \eqref{kle} and the highest degree $K$ of orthogonal polynomials used in the chaos expansion and is given by 
\[ 
	M+1 = \frac{(N+K)!}{N!K!}  . 
\]
Under suitable assumptions the expansion converges in $L_2$ as a consequence of the  Cameron-Martin theorem \cite{cameron1947}. A detail convergence analysis of the method is outside the scope of this paper.
 
Thus, the approximation of the equation \eqref{spde_pce} leads to solving 
\begin{align*}
	\sum_{k = 0}^{M} \dot{X}^M_k(t) H_k(\omega) = \mathcal{A} \sum_{k = 0}^{M} X^M_k(t) H_k(\omega)  + \sum_{j = 1}^{N} \sqrt{\lambda_j} \varphi_j(t) \eta_j  .
\end{align*}
Performing a Galerkin projection onto each polynomial $H_i$ leads to
\[ 
	\dot{X}^M_i(t) \langle H_i, H_i \rangle = \mathcal{A} X^M_i(t) \langle H_i,H_i \rangle + \sum_{j = 1}^{N} \sqrt{\lambda_j} \varphi_j(t) \langle  \eta_j, H_i \rangle  \, , \qquad i = 0, \ldots M  , 
\]
since $H_i, i = 0, \ldots M$ are orthogonal.
We have obtained a deterministic system of $M+1$ differential equations. Any numerical method can be used for solving these systems of equations, here we use the Crank Nicolson method.  

In the same setting the chaos expansion method can be applied to the SPDEs based models with multiplicative noise and the SDEs based models described in Section 2.

To illustrate the good behavior of the method, in the following plot one can see the order of convergence of the mean and the covariance of the multiplicative SPDE. 


\begin{figure}[H]
	\centering
	\includegraphics[trim = 10px 250px 10px 220px, clip, width=0.7\textwidth]{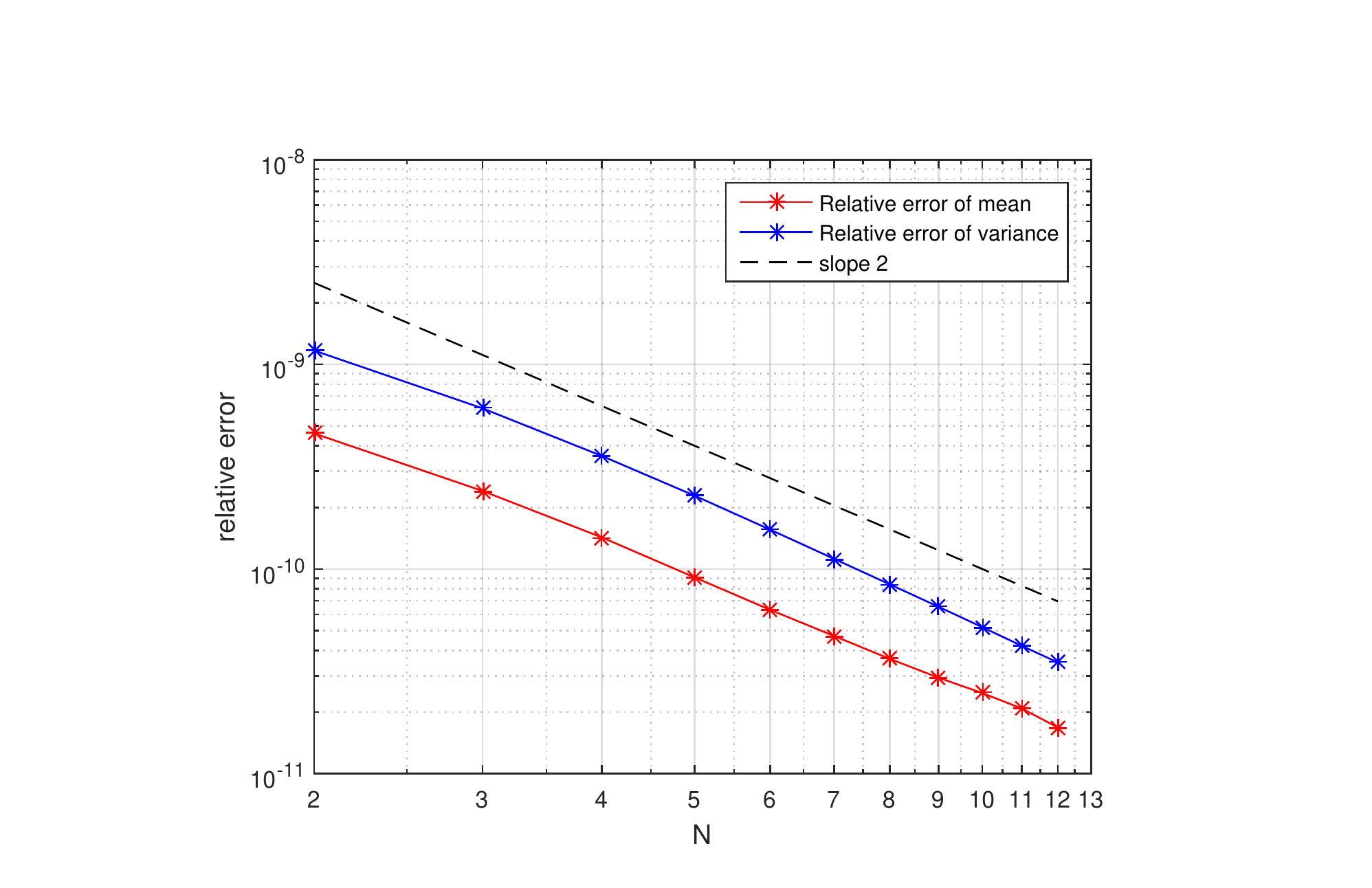}
	\caption{The relative error is computed by $\frac{\Vert X_N-X_{N-1}\Vert}{\Vert X_N\Vert }$, where we denote with $X_N$ the solution obtained by $N$ random variables. }
	\label{fig:pce2}
\end{figure}  

The chaos expansion method works very well for this example as a small error is achieved using only few random variables. Thus, for the computations in next section we choose the parameter $N=3$. 

In the following section we will the show numerical results. All the algorithms and methods described in this section were implemented in \texttt{MATLAB R2015a}.

	\section{Numerical simulations}

We perform numerical simulations of sea surface temperature in various scenarios and validate the results with measured data for a year with an El Ni\~no (2015) event and for a year without (2014). We compare a direct approach proposed in this paper with stochastic differential equations and SPDEs based models solved by Taylor methods and stochastic Galerkin methods, respectively. 

Following \cite{penland1995} we construct the sea surface temperature (SST) anomalies, using the dataset OISST \cite{cisl}, whereas for the ocean currents we use the dataset OSCAR \cite{OSCAR}, both from the National Oceanic and Atmospheric Administration (NOAA). As the most interesting area is the Indo-Pacific ocean, we compute the data on the rectangle $[30,290] \times [-30,30]$ (i.e. Indo-Pacific ocean) where the unit is given by Longitude and Latitude respectively.   

In order to compare the models with already measured data, the dataset consisting of SST anomalies is truncated after $13$th June $2014$ and $2015$ respectively and the next months are predicted. 
As initial data the SST anomalies at the $14$th June are taken. We choose as step size  $0.5$ days and make simulations for the next $200$ days.

We compare the SPDE models itselfs as well as the different numerical solvers presented in the previous section with the measured data. The difference between measured data and simulations is computed by 
\begin{align}
\text{err}(t) = \frac{1}{n }\sum_{k = 1}^{n} (X(t)-\tilde{X}_k(t)) , \label{err_est} 
\end{align}  
where $X$ is the measured data, $\tilde{X}_k$ the k-th simulation of the numerical scheme and $n$ is the number of realizations, in the following we choose $n = 50$.  The error is computed in the rectangle $[160,270] \times[-5,5]$ as the SST anomalies in this area are mainly affected by the El Ni\~no phenomenon (see Figure \ref{fig:meas2015}). 

We compute a numerical solution of the SDE models with the Taylor $1.5$ scheme \cite{kloeden2011} to show that the SPDE models work better than the SDE models. The numerical solutions of the equations are computed up to $T = 200$ and  the relative difference between the measured data and the realizations of the stochastic equations obtained by the estimate \eqref{err_est} is calculated. 

\begin{figure}[H]
	\begin{minipage}[t]{0.487\textwidth}
		\centering
		\includegraphics[width=1\textwidth]{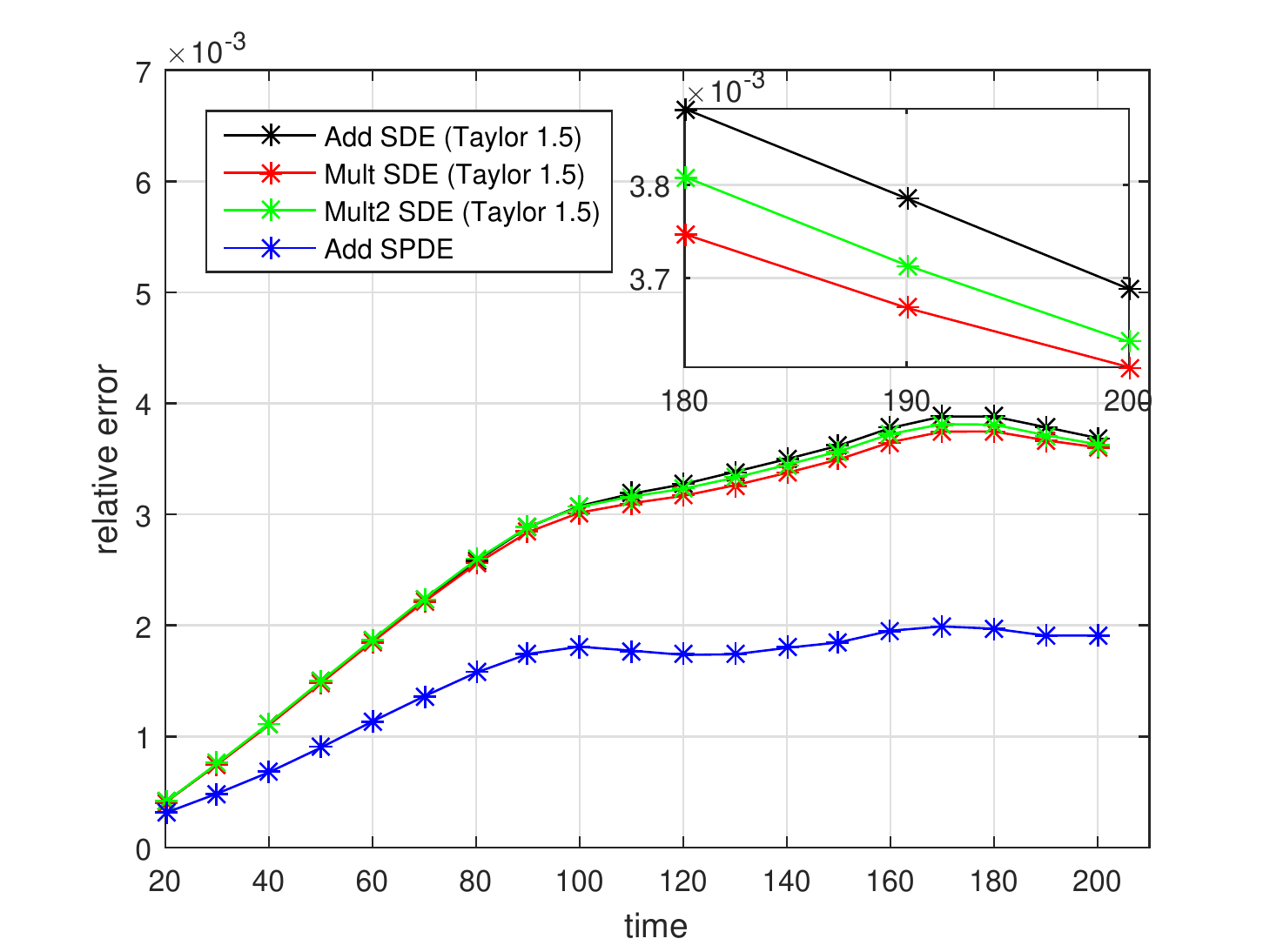}
		\caption{Relative difference between measured data and the realizations for the year $2014$ (no El Ni\~no year)  }
		\label{fig:simu2014sde}
	\end{minipage}\hfill
	\begin{minipage}[t]{0.487\textwidth}
		\centering
		\includegraphics[width=1\textwidth]{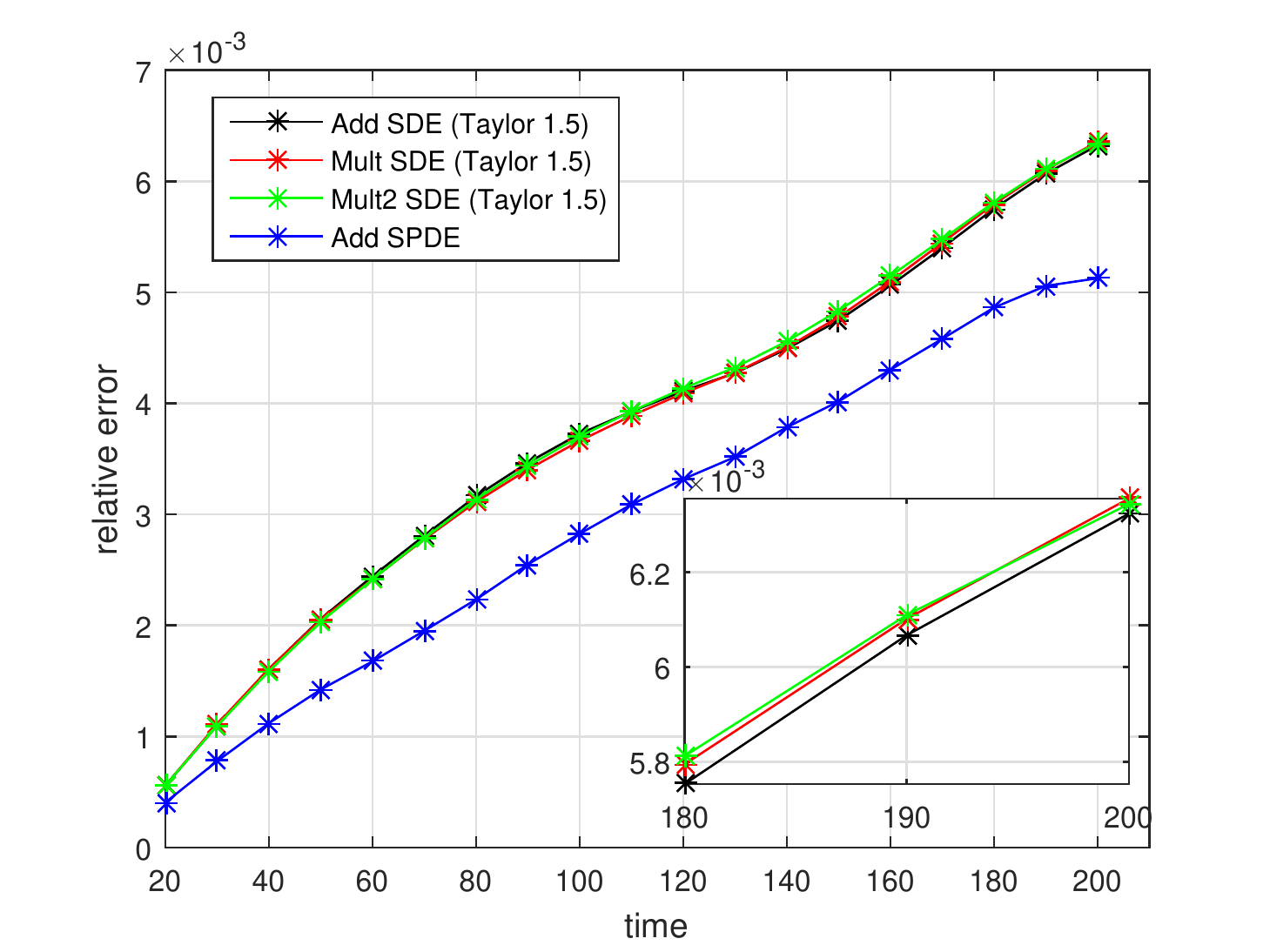}
		\caption{Relative difference between measured data and the realizations for the year $2015$ (El Ni\~no year)  }
		\label{fig:simu2015sde}
	\end{minipage}\hfill
\end{figure}

We further compare the stochastic Galerkin method (in the figures denoted by S.Galerkin) with the deterministic approach for mean and covariance described in Section 4 (in the figures denoted by mean\&cov) . 
\begin{figure}[H]
	\begin{minipage}[t]{0.487\textwidth}
	\centering
	\includegraphics[width=1\textwidth]{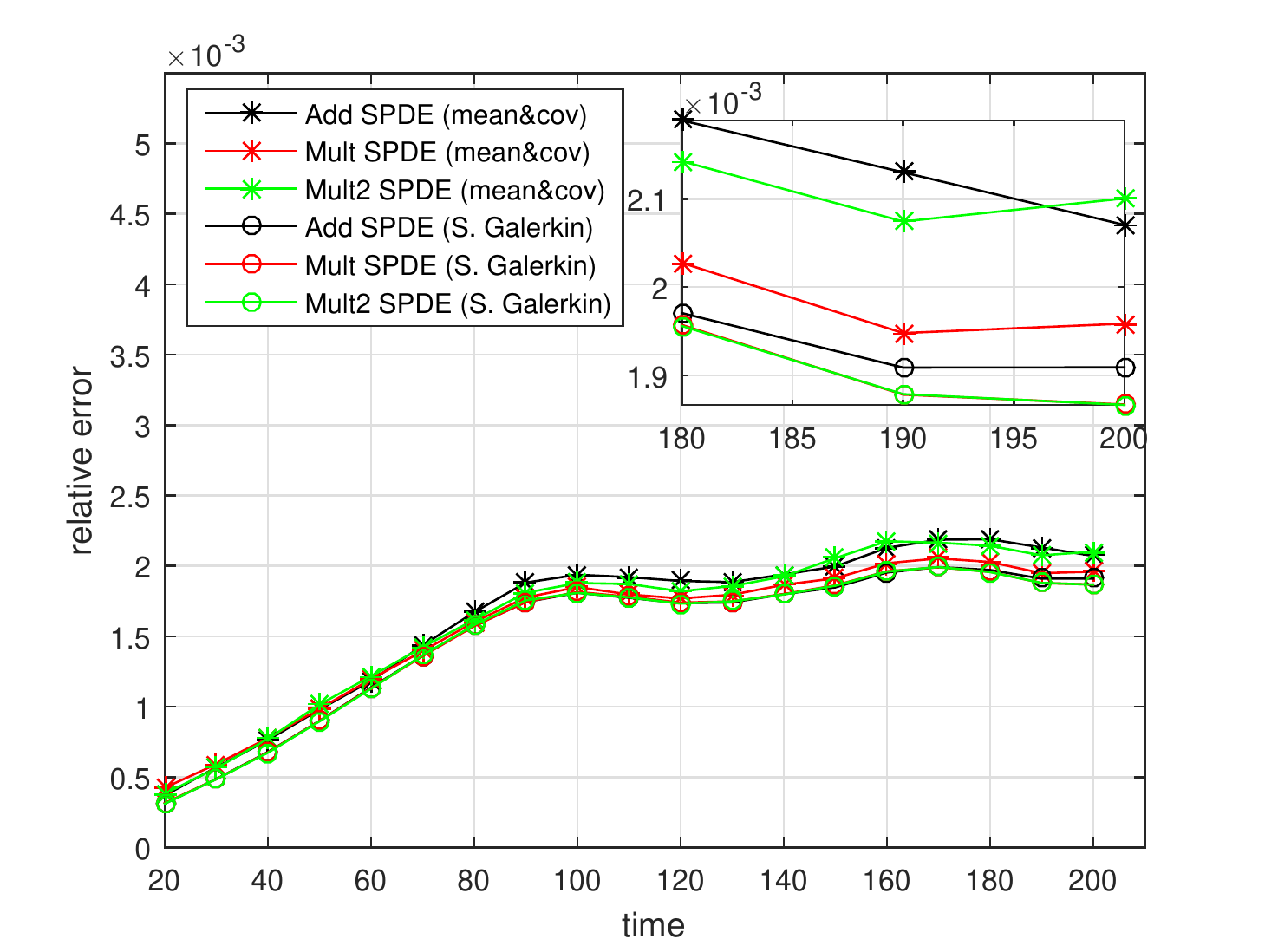}
	\caption{Relative difference between measured data and the realizations for the year $2014$ (no El Ni\~no year)  }
	\label{fig:simu2014}
	\end{minipage}\hfill
	\begin{minipage}[t]{0.487\textwidth}
		\centering
		\includegraphics[width=1\textwidth]{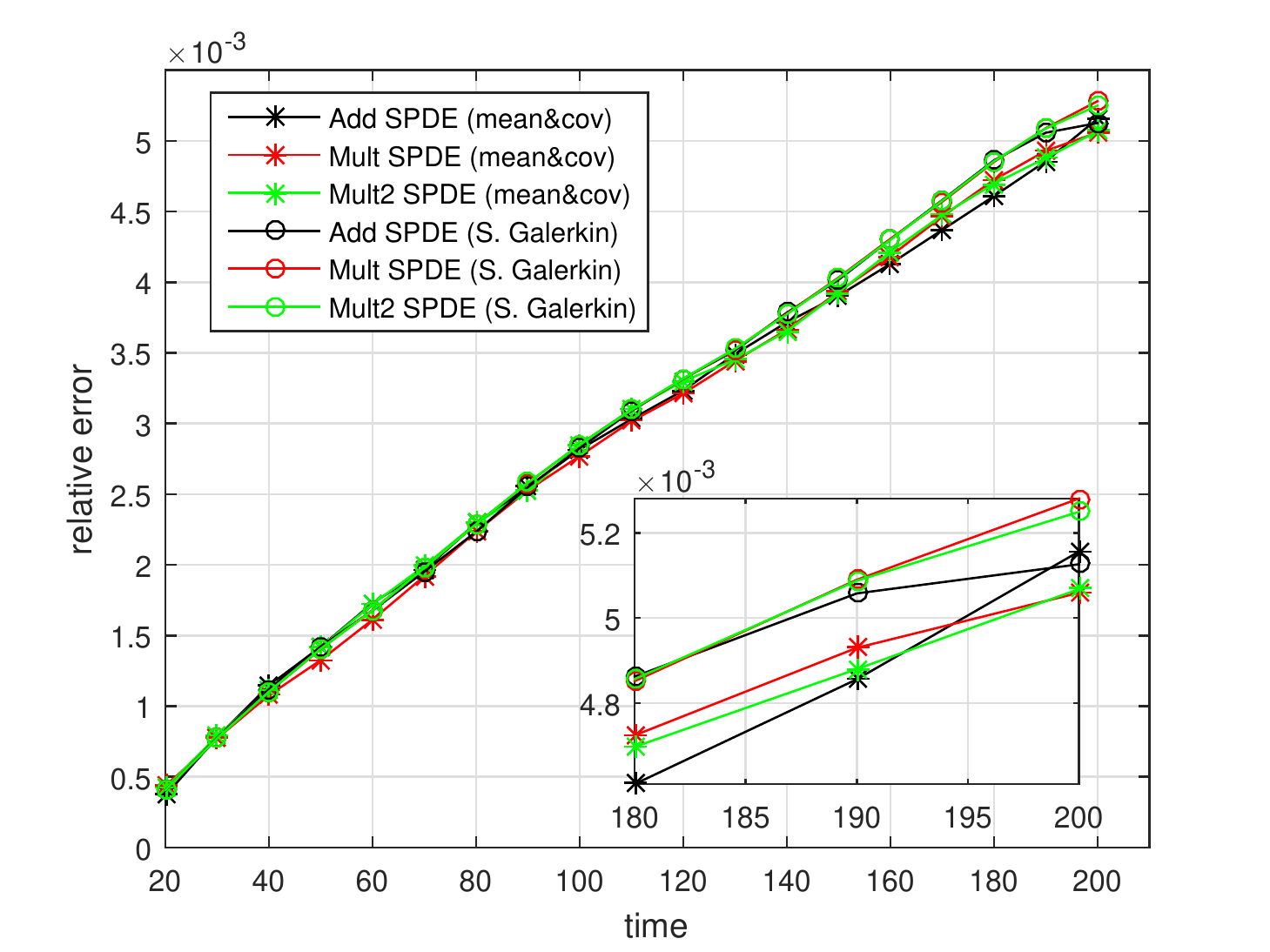}
	\caption{Relative difference between measured data and the realizations for the year $2015$ (El Ni\~no year)  }
		\label{fig:simu2015}
	\end{minipage}\hfill
\end{figure}

From Figures \ref{fig:simu2014sde} and \ref{fig:simu2015sde} we see that the SPDE models lead to better results than the SDE models, which is expected as the SDE model is a simplification of the SPDE model.  Moreover, the approach with solving first the deterministic equations for mean and covariance and compute afterwards a realization with the computed mean and covariance works even better than the approach obtained by polynomial chaos in a year with an El Ni\~no event. Also the multiplicative models seem to be more appropriate than the additive model. As suspected, the error in a year without El Ni\~no is less than in a year with El Ni\~no.

Additionally, we further compute the computational time of the algorithms for the additive and multiplicative SPDE. Therefore, we perform $N = 100$ steps with the algorithms for different sizes of the grid and measured the time.

\begin{figure}[H]
	\centering
	\includegraphics[width=0.5\textwidth]{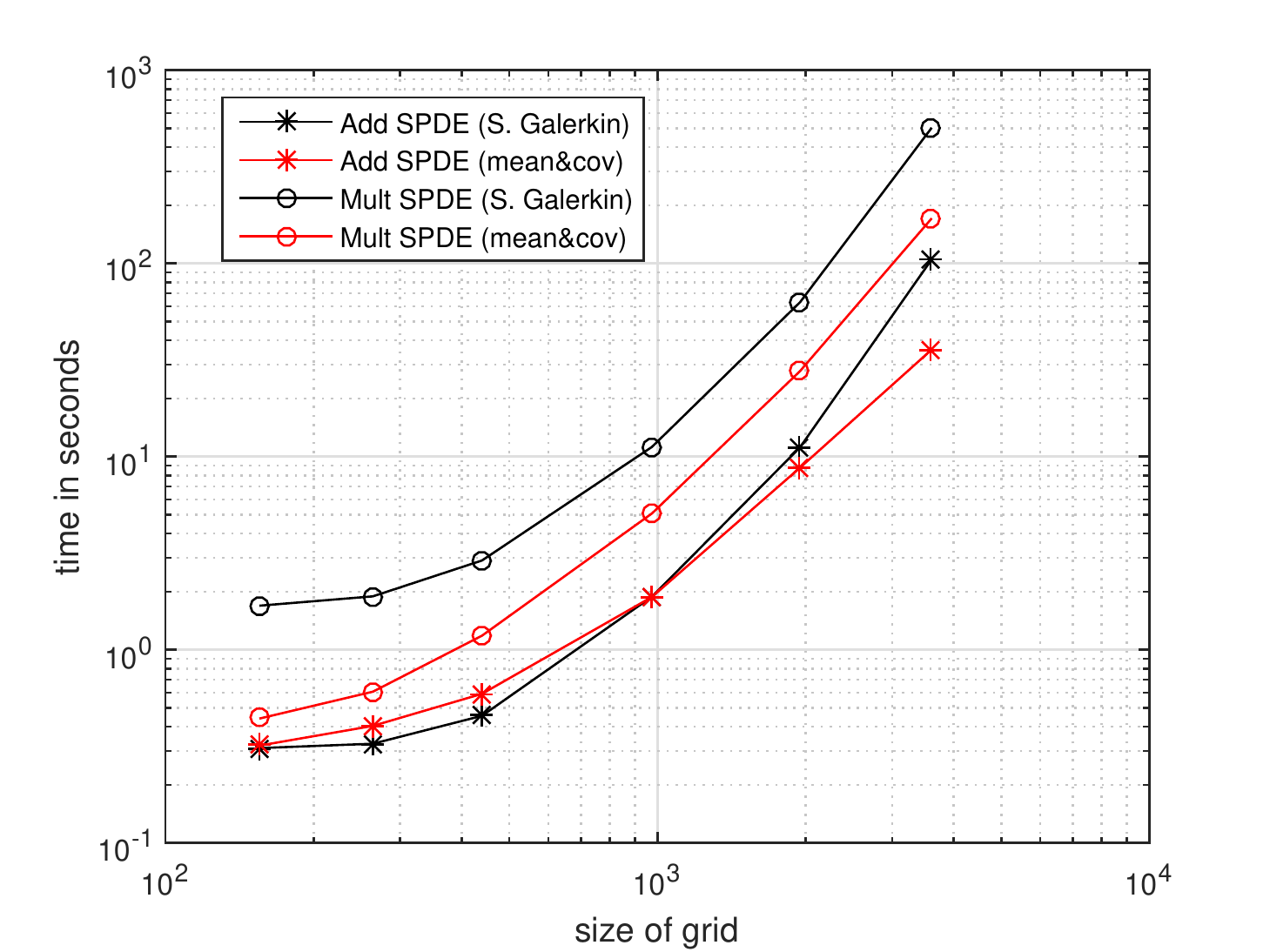}
	\caption{Computational costs for the additive and multiplicative SPDE. We compare the stochastic Galerkin methods (in the figure S.Galerkin) with the approach for mean and covariance (in the figure mean\&cov)}
	\label{fig:compcost}
\end{figure}

Especially when using a fine grid resolution and hence solving a large-scale problem that describes better the phenomenon, our method seems quite promising. For the multiplicative model we are also faster for coarse grids. Another advantage of our method is that it uses less memory, as the splitting algorithm for the covariance equation does not compute the full solution $P$ but the low-rank factorization $P_1P_1\trasp$. For the stochastic Galerkin method however, we have to set up the system of deterministic equations, which can, depending on the number of random variables and orthogonal basis functions, very large.  The proposed method is in general computationally less expensive and it has a great potential not only for a simulation of El Ni\~no but for any physical phenomenon modeled by linear SPDEs.
  
In Figures \ref{fig:meas2015} and \ref{fig:meas2014} one can see one realization of the additive model for the 11th December 2015 and the 11th December 2014, respectively. Moreover, as a comparison, the measured data is plotted in Figures \ref{fig:Pred2015} and \ref{fig:Pred2014}. The SST anomalies are given in degree Celsius. We show a section of the Indo-Pacific ocean with Africa being on the left part and America on the right part. The black part indicates the land, bright colors indicate higher temperatures and dark ones lower temperatures than usual. From Figure \ref{fig:Pred2015} we observe the typical face of an El Ni\~no event with the unusual warming of the sea surface, although we are not able to predict the strength of the 2015 El Ni\~no event, see Figure \ref{fig:meas2015}. However, in 2015 was one of the strongest El Ni\~no events in the past decades, repeating the simulations in a year with a weaker El Ni\~no, the results get better. Moreover, we see a completely different behaviour of the SST anomalies in a year without an event, see Figures \ref{fig:Pred2014} and \ref{fig:meas2014}. In these Figures, the temperature anomalies are nearly uniformly distributed and no unusual warming is measured and predicted.  

\begin{figure}[H]
	\centering
	\includegraphics[trim = 15px 290px 15px 280px, clip, width=1.1\textwidth]{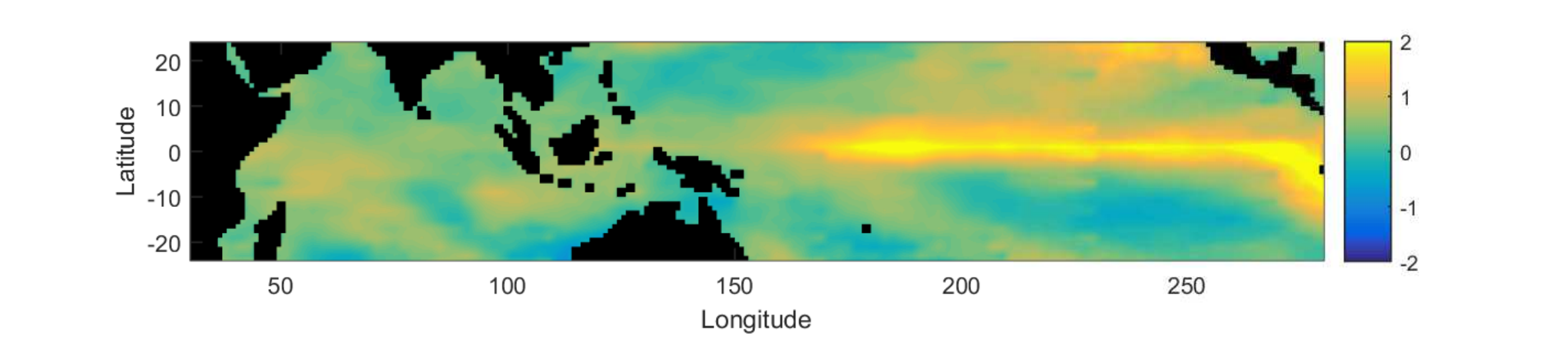}
	\caption{Predicted SST anomalies in December 2015}
	\label{fig:Pred2015}
\end{figure}

\begin{figure}[H]
	\centering
	\includegraphics[trim = 15px 290px 15px 280px, clip, width=1.1\textwidth]{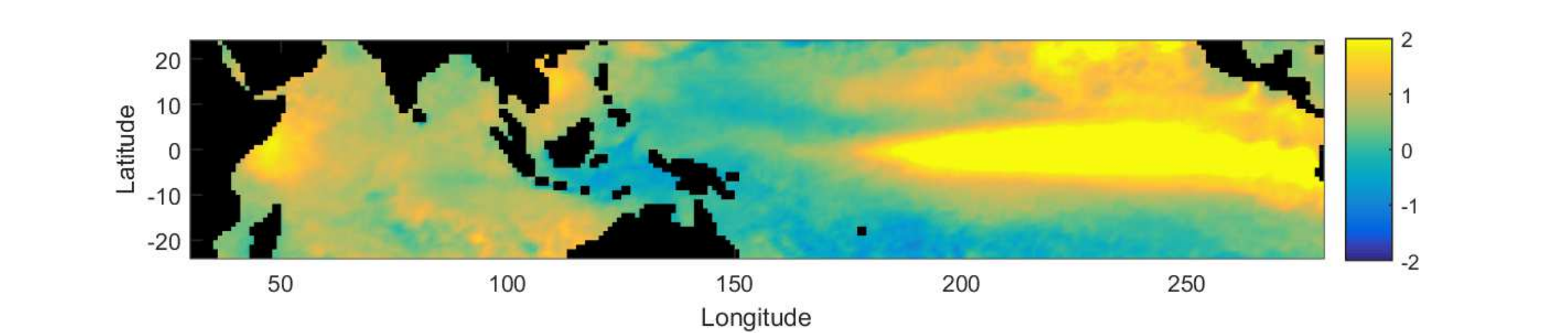}
	\caption{Measured SST anomalies in December 2015}
	\label{fig:meas2015}
\end{figure}  

\begin{figure}[H]
	\centering
	\includegraphics[trim = 15px 290px 15px 280px, clip, width=1.1\textwidth]{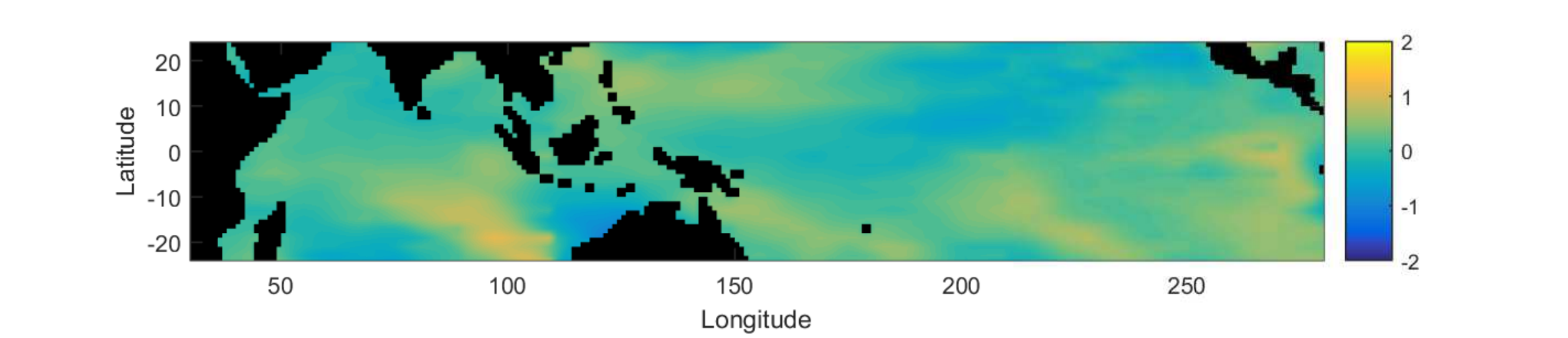}
	\caption{Predicted SST anomalies in December 2014}
	\label{fig:Pred2014}
\end{figure}

\begin{figure}[h]
	\centering
	\includegraphics[trim = 15px 290px 15px 280px, clip, width=1.1\textwidth]{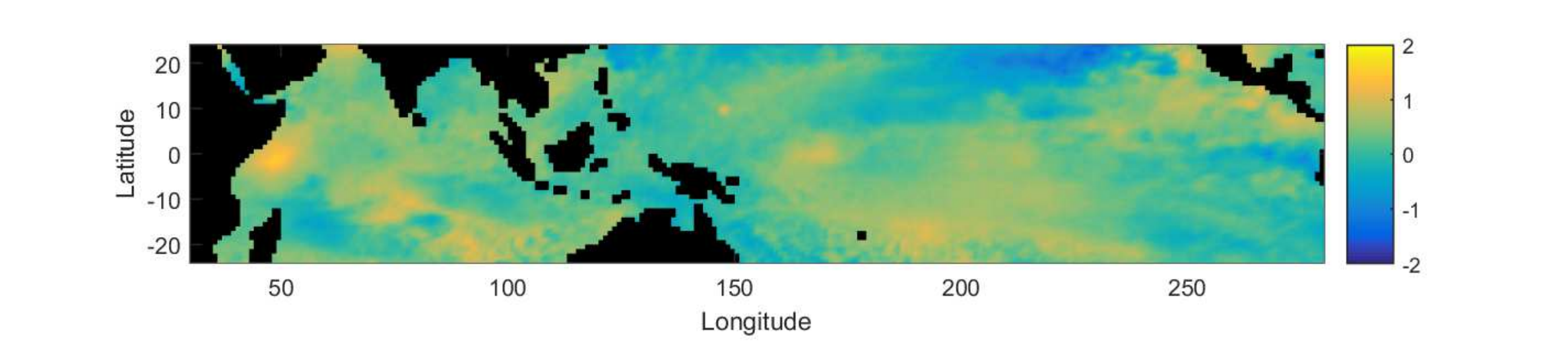}
	\caption{Measured SST anomalies in December 2014}
	\label{fig:meas2014}
\end{figure}  
	\section{Conclusion}

We proposed an alternative way of solving  numerically linear stochastic partial differential equations that model El Ni\~no. This approach relies on the solution of deterministic equations for mean and covariance. We  compared this method with models from the literature solved by standard numerical schemes. Particularly, for years with an El Ni\~no event, this approach seems to be very promising as the computational cost  is lower. The method has a great potential as it can be applied to any linear stochastic differential equation whose solution is a Gaussian process. In addition, we discussed modifications of existing models for El Ni\~no and showed that in general SPDE based models perform better than the SDE models from  literature.

	\section*{Acknowledgements}
	This work was supported by the Austrian Science Fund (FWF) -- project id: P27926.

	\bibliographystyle{elsarticle-num} 
	\nocite{*}
	\bibliography{MP17a}

\end{document}